\documentclass[12pt, noamsfonts, reqno]{amsart}

\usepackage[margin=2.8cm]{geometry}
\usepackage{booktabs}
\usepackage[USenglish]{babel}
\usepackage{microtype}
\usepackage{mathtools}
\mathtoolsset{showonlyrefs}

\usepackage[T1]{fontenc}
\usepackage[
  theoremfont,
  varg,
  vvarbb,
  smallerops,
  slantedGreek,
]{newtx}

\usepackage{tikz-cd}

\usepackage{hyperref}
\usepackage{orcidlink}
\usepackage[msc-links]{amsrefs}

\theoremstyle{plain}
\newtheorem{proposition}{Proposition}
\newtheorem{corollary}[proposition]{Corollary}
\newtheorem{theorem}[proposition]{Theorem}
\newtheorem{lemma}[proposition]{Lemma}
\newtheorem*{MT}{Main Theorem}

\theoremstyle{definition}
\newtheorem{definition}{Definition}
\newtheorem{example}{Example}

\theoremstyle{remark}
\newtheorem{remark}{Remark}

\DeclareMathOperator{\J}{J}
\DeclareMathOperator{\Aff}{Aff}
\DeclareMathOperator{\Aut}{Aut}
\DeclareMathOperator{\End}{End}
\DeclareMathOperator{\im}{Im}
\DeclareMathOperator{\Irr}{Irr}
\DeclareMathOperator{\pr}{P}
\DeclareMathOperator{\GL}{GL}
\DeclareMathOperator{\Gal}{Gal}
\DeclareMathOperator{\s}{s}
\DeclareMathOperator{\tr}{Tr}
\DeclareMathOperator{\Hom}{Hom}

\DeclareMathOperator{\Ind}{Ind}
\DeclareMathOperator{\Stab}{Stab}
\DeclareMathOperator{\Fix}{Fix}
\DeclareMathOperator{\id}{id}
\DeclareMathOperator{\nd}{\sigma_0}
\DeclareMathOperator*{\lcm}{lcm}

\newcommand{\GF}[1]{\mathbb{F}_{#1}}
\newcommand{\CC}{\mathbb{C}}
\newcommand{\QQ}{\mathbb{Q}}
\newcommand{\ZZ}{\mathbb{Z}}
\newcommand{\tra}{T}
\newcommand{\lin}{\Lambda}
\newcommand{\st}{\theta}
\newcommand{\ppi}{\piup}
\newcommand{\Sym}{\mathfrak{S}}
\newcommand{\GA}[1]{\Aff(\GF{#1})}

\DeclarePairedDelimiter{\card}{\lvert}{\rvert}
\DeclarePairedDelimiter{\ord}{\lvert}{\rvert}
\DeclarePairedDelimiter{\gen}{\langle}{\rangle}
\DeclarePairedDelimiterX{\inpr}[2]{\langle}{\rangle}{#1,#2}
\DeclarePairedDelimiterX{\extdeg}[2]{[}{]}{#1:#2}
\DeclarePairedDelimiterX{\ind}[2]{[}{]}{#1:#2}
\DeclarePairedDelimiter{\paren}{(}{)}

\begin{document}

\title[Group algebra decompositions not affordable by Prym varieties]{Jacobian
  varieties with group algebra decomposition not affordable by Prym varieties}

\author[Benjamín M. Moraga]{Benjamín M.\@ Moraga
  \orcidlink{0000-0003-3211-0637}}

\date{February 10th, 2024}

\email{\href{mailto:benjamin.baeza@ufrontera.cl}{benjamin.baeza@ufrontera.cl}}

\thanks{Partially supported by ANID Fondecyt grants 1190991.}

\address{Departamento de Matemática y Estadística, Universidad de La Frontera,
  Temuco, Chile}


\begin{abstract}
  The action of a finite group \(G\) on a compact Riemann surface \(X\)
  naturally induces another action of \(G\) on its Jacobian variety
  \(\operatorname{J}(X)\). In many cases, each component of the group algebra
  decomposition of \(\operatorname{J}(X)\) is isogenous to a Prym varieties of
  an intermediate covering of the Galois covering \(\pi_G\colon X \to X/G\); in
  such a case, we say that the group algebra decomposition is affordable by
  Prym varieties. In this article, we present an infinite family of groups that
  act on Riemann surfaces in a manner that the group algebra decomposition of
  \(\operatorname{J}(X)\) is not affordable by Prym varieties; namely, affine
  groups \(\operatorname{Aff}(\mathbb{F}_q)\) with some exceptions: \(q = 2\),
  \(q = 9\), \(q\) a Fermat prime, \(q = 2^n\) with \(2^n-1\) a Mersenne prime
  and some particular cases when \(X/G\) has genus \(0\) or \(1\). In each one
  of this exceptional cases, we give the group algebra decomposition of
  \(\operatorname{J}(X)\) by Prym varieties.
\end{abstract}

\maketitle


\section{Introduction}
\label{sec:introduction}

Let \(X\) be a compact Riemann surface and let \(G\) be a finite group acting
faithfully on \(X\); namely, there is a monomorphism \(G \to \Aut(X)\) (we denote
the images of this map just as elements of \(g\)). Set \(Y \coloneq X/G\) so the
quotient map \(\pi_G \colon X \to Y\) is a Galois covering between compact Riemann
surfaces. The action of \(G\) on \(X\) naturally induces an action on the
Jacobian variety \(\J(X)\) of the curve \(X\) and, moreover, this action
induces an homomorphism \(\QQ[G] \to \End_{\QQ}(\J(X))\), where
\(\End_{\QQ}(\J(X)) \coloneq \End(\J(X)) \otimes_{\ZZ} \QQ\), which yields a
\(G\)-equivariant isogeny
\[
  \J(X) \sim B_1^{n_1} \times \cdots \times B_r^{n_r}
\]
called the \emph{group algebra decomposition} of \(\J(X)\) (we give more
details about this decomposition in section~\ref{sec:background}, for a broader
discussion see \cite{book:lange22}*{section~2.9}). Each group algebra component
\(B_i\) corresponds to a rational irreducible representation of \(G\), if we
let \(B_1\) corresponds to the trivial representation, then \(B_1 \sim
\J(Y)\). The other group algebra components, namely \(B_i\) for
\(i=2,\ldots,r\), may sometimes be described as Prym varieties of intermediate
coverings of \(\pi_G\) up to isogeny, we discuss this phenomenon in
section~\ref{sec:prym-varieties}. Its foundational example is studied in
\cite{art:recillas1974}, where the Jacobian of a tetragonal curve is isomorphic
to the Prym variety of a double cover of a trigonal curve, this is called the
Recillas trigonal construction. There are several cases in which all the group
algebra components are isogenous to Prym varieties of intermediate coverings,
in such a case we will say that the group algebra decomposition is affordable
by Prym varieties (see Definition~\ref{def:affordable}). The general theory of
this phenomenon was first discussed in \cite{art:carocca2006} and was latter
summarized in \cite{book:lange22} with an extensive study of the group algebra
decomposition for the Galois closure of coverings of degree \(2\), \(3\) and
\(4\); some families of cyclic and dihedral groups where also studied. In all
these examples, the group algebra decomposition is affordable by Prym
varieties.  In \cite{art:moraga23}, the Galois closure of a fivefold covering
was studied, for the particular case where \(G = \Sym_5\), the symmetric group
of degree \(5\), there is a group algebra component which is not isogenous to
the Prym variety of any intermediate covering (in \cite{art:lange2004} the Prym
variety of a \emph{pair of coverings} is defined in order to describe this
component). It is also worth mentioning that in
\cite{art:carocca2006}*{appendix~B} an example is given in which two group
algebra components are neither a Prym variety of an intermediate covering nor
an intersection of two of them.

In this article we give an infinite family of finite groups, the full one
dimensional affine groups \(\GA{q}\) over a Galois \(\GF{q}\) field with
\(q = p^n\) and \(p\) prime, such that, for most of them, the group algebra
decomposition is not affordable by Prym varieties. To be precise, we now state
the main result of this work, proven in section~\ref{sec:galois-covers}.

\begin{MT}[Theorem~\ref{thm:main}]
  The group algebra decomposition of \(\J(X)\) is affordable by Prym varieties
  if and only if at least one of the following three conditions is met:
  \begin{enumerate}
  \item \label{item:mt-1} The integer \(q\) is equal to \(2\) or \(9\), is a
    Fermat prime or \(q-1\) is a Mersenne prime.
  \item \label{item:mt-2} The signature of \(\pi_G\) is of the form
    \begin{equation}
      (1;p,\ldots,p)\ \text{or}\ (0;p,\ldots,p,q-1,q-1).
    \end{equation}
  \item \label{item:mt-3} The integer \(q-1\) is equal to \(d^\mu e^\nu\), where
    \(d\) and \(e\) are different prime numbers and \(\mu\) and \(\nu\) are
    positive integers, and the signature of \(\pi_G\) is of the form
    \begin{equation}
      \begin{gathered}
        (1;p,\ldots,p,d,\ldots,d), (0;p,\ldots,p,q-1,q-1,d,\ldots,d), \\
        (0;p,\ldots,p,q-1, e^\nu, d, \ldots, d)\ \text{or}\ (0;p,\ldots,p,e^\nu, e^\nu, d, \ldots, d)
      \end{gathered}
    \end{equation}
    (the last two signatures are only possible if \(\mu = 1\)).
  \end{enumerate}
\end{MT}
Note that in the cases enumerated in items~\ref{item:mt-2}~and~\ref{item:mt-3}
the curve \(Y\) is either the Riemann sphere or a torus, and, the ramification
of \(\pi_G\) is very restricted; that is, for most signatures the group algebra
decomposition is not affordable by Prym varieties. With respect to
item~\ref{item:mt-1}, recall that the only known Fermat primes are the first
five Fermat numbers and, until 2024, only forty-eight Mersenne primes are
confirmed (see \url{https://www.mersenne.org/primes/}). Summarizing, for most
values of \(q\), the group algebra decomposition of \(\J(X)\) is not affordable
by Prym varieties. For example, setting \(q \coloneqq p^n\) for any \(p > 2\)
yields an infinite family of groups with that property. In the cases where the
decomposition is affordable by Prym varieties, we give it in
Corollary~\ref{cor:MT}.

In section~\ref{sec:background}, we fix notation and state some properties on
rational irreducible representations, the group algebra decomposition, Galois
coverings and Prym varieties. Section~\ref{sec:irrC} is about complex
irreducible representations of \(\GA{q}\). They were first inductively
classified in \cite{art:faddeev1976} and their characters were constructed,
based on this classification, in \cite{art:siegel1992}. Also, some particular
cases where explicitly computed in \cite{art:iranmanesh1997}. However, no
explicit character table for dimension one was found on the literature and,
being it easily computable through the Wigner--Mackey method, we compute it
explicitly. Nevertheless, compare Theorem~\ref{thm:comp-rep-aff} and
\cite{art:siegel1992}*{Theorem~3.11}. In section~\ref{sec:irrQ}, we describe
the rational irreducible representations of \(\GA{q}\) and in
section~\ref{sec:subgroups} we classify its subgroups. Finally, in
section~\ref{sec:galois-covers} we prove our main result.

\section{Background}
\label{sec:background}

\subsection{Rational irreducible representations}
\label{sec:rational-rep}

For the whole of this section, let \(G\) denote a finite group. Let
\(\varsigma \colon G \to \GL(V)\) be a complex irreducible representation of
\(G\), \(L\) its \emph{field of definition} and \(K\) its \emph{character
  field}, so we have \(K\subseteq L\). The integer
\(\s(\varsigma) \coloneq \extdeg{L}{K}\) is called the \emph{Schur index} of the
representation (see \cite{book:serre77}*{section~12.2}). According to
\cite{book:lange22}*{equation~(2.19)}, there is a unique (up to isomorphism)
rational irreducible representation \(\rho\colon G\to \GL(W)\) such that
\begin{align}
  \label{eq:rat-decomp}
  \rho \otimes \CC &\cong \s(\varsigma) \bigoplus_{\mathclap{\gamma \in \Gal(K/\QQ)}} \varsigma^\gamma\\
\intertext{and hence}
  \label{eq:rat-char}
  \chi_\rho(g) &= \s(\varsigma) \tr_{K/\QQ} (\chi_\varsigma(g)),
\end{align}
where \(\chi_\varphi\) denotes the character of a representation \(\varphi\). For simplicity
we will not distinguish between \(\rho\) and \(\rho \otimes \CC\).

\begin{definition}
  We say that the representations \(\rho\) and \(\varsigma\) are \emph{Galois associated}.
\end{definition}

Set \(\Irr_{\QQ}(G) \coloneq \{\rho_1, \ldots, \rho_r\}\), the set of irreducible rational
representations of \(G\) up to isomorphism. To each representation \(\rho_i\)
corresponds a unique central idempotent of the group-algebra \(\QQ[G]\) denoted
by \(e_i\) such that \(\QQ[G]e_i\) affords \(\rho_i\) and
\(e_1+\cdots+e_r = 1\). These idempotents induces a unique decomposition of
\(\QQ[G]\) into simple sub-algebras
\begin{equation}
  \label{eq:iso-qg}
  \QQ[G] = \QQ[G]e_1 \oplus \cdots \oplus \QQ[G]e_r
\end{equation}
called \emph{isotypical decomposition} of \(\QQ[G]\) (see
\cite[equation~2.27]{book:lange22}). Moreover, each simple sub-algebra
\(\QQ[G]e_i\) can be (non-uniquely) decomposed into
\(n_i \coloneq \deg(\varsigma_i)/\s(\varsigma_i)\) indecomposable left \(\QQ[G]\)-modules, where
\(\varsigma_i\) is a complex irreducible representation Galois associated to
\(\rho_i\). More precisely, there are primitive orthogonal idempotents
\(q_{i,1}, \ldots, q_{i,n_i}\) such that \(e_i = q_{i,1} + \cdots + q_{i,n_i}\) (see
\cite[equation~2.31]{book:lange22}). This idempotents induce the \emph{group
  algebra decomposition}:
\begin{equation}
  \label{eq:ga-qg}
  \QQ[G] = \bigoplus_{i=1}^r \bigoplus_{j=1}^{n_i} \QQ[G]q_{i,j}.
\end{equation}

We now introduce a particular kind of representations that will be important in
section~\ref{sec:prym-varieties}. For a subgroup \(H \subset G\) we set
\(\rho_H \coloneq \Ind_H^G (1_H)\), where \(1_H\) denotes the trivial representation of
\(H\). Since \(1_H\) is a rational representation, \(\rho_H\) is also
rational.

\begin{remark}
  \label{rem:Frobenius}
  The multiplicity of the irreducible components of \(\rho_H\) may be computed
  from the character table of \(G\) and the cardinality of the \(G\)-conjugacy
  classes of \(H\) through Frobenius reciprocity (see
  \cite{book:serre77}*{Theorem~13}).
\end{remark}

\subsection{Group algebra decomposition of an abelian variety}
\label{sec:ga-jac}

Consider an abelian variety \(A\) with a \(G\)-action. The action of \(G\)
induces an algebra homomorphism
\(\QQ[G] \to \End_{\QQ}(A) \coloneqq \End(A) \otimes_{\ZZ} \QQ\); even if it is not a
monomorphism, we do not distinguish between elements of \(\QQ[G]\) and their
images in \(\End_{\QQ}(A)\). Each element \(\gamma \in \QQ[G]\) defines an abelian
subvariety \(\im(\gamma) \coloneqq \im(a\gamma)\), where \(a\) is any integer such that
\(a\gamma \in \End(X)\). This definition does not depend on the choice of
\(a\). The isotypical decomposition of \(\QQ[G]\) of equation~\eqref{eq:iso-qg}
induces a decomposition of \(A\) as stated in the following result.

\begin{proposition}[\cite{book:lange22}*{Theorem~2.9.1}]
  \label{prop:isotypical}
  \hfill
  \begin{enumerate}
  \item \label{item:hom} The subvarieties \(\im(e_i)\) are \(G\)-stable with
    \(\Hom_G(\im(e_i), \im(e_j)) = 0\) for \(i \neq j\).
  \item The addition map induces a \(G\)-equivariant isogeny
    \[ \im(e_1) \times \cdots \times \im(e_r) \sim A, \]
    where \(G\) acts on \(\im(e_i)\) by the representation \(\rho_i\).
  \end{enumerate}
\end{proposition}

This isogeny is the \emph{isotypical decomposition} of \(A\). Moreover, the
decomposition in equation~\eqref{eq:ga-qg} yields the following result.

\begin{proposition}[\cite{book:lange22}*{Theorem~2.9.2}]
  \label{prop:ga-ab}
  With the previous notation, there are abelian subvarieties \(B_1,\ldots,B_r\) of
  \(A\) and a \(G\)-equivariant isogeny
  \begin{equation}
    \label{eq:ga-ab}
    B_1^{n_1} \times \cdots \times B_r^{n_r} \sim A,
  \end{equation}
  where \(G\) acts on \(B_i^{n_i}\) via the representation \(\rho_i\) (with
  appropriate multiplicity).
\end{proposition}

\begin{remark}
  We can set \(B_i \coloneq \im(q_{i,1})\) for \(i=1,\ldots,r\), for example, but the
  components \(B_i\) are not uniquely determined.
\end{remark}

\begin{definition}
  The isogeny of equation~\eqref{eq:ga-ab} is the \emph{group algebra
    decomposition} of \(A\) and each \(B_i\) is a \emph{group algebra
    component} of \(A\).
\end{definition}

\subsection{Galois coverings}
\label{sec:GalCov}

Let \(X\) be a compact Riemann surface with a \(G\)-action; that is, there is a
monomorphism from \(G\) to the automorphism group \(\Aut(X)\) of \(X\). This
action defines a \emph{Galois covering} \(\pi_G \colon X \to Y\), where
\(Y\coloneq X/G\) (see \cite{book:lange22}*{section~3.1}). For a pair of subgroups
\(H\subset N\) of \(G\), the induced maps \(\pi^H_N\colon X/H \to X/N\) are called
\emph{intermediate coverings} of \(\pi_G\). The genus and ramification data of
the intermediate coverings are determined by the \emph{geometric signature}
\begin{equation}
  \label{eq:geometric_signature}
  (g; G_1, \cdots, G_s)
\end{equation}
of \(\pi_G\); that is, the genus \(g\) of \(Y\) and the stabilizer subgroups
\(G_j\) of \(G\) with respect to branch points in each different orbit of
branch points of \(\pi_G\); the subgroups \(G_j\) are determined up to conjugacy
and reenumeration. The \emph{signature} of the \(G\)-action on \(X\) is
\((g; n_1,\ldots,n_s)\), where \(n_i \coloneq \card{G_i}\) for \(i=1,\ldots,s\).

With respect to the existence of a Galois covering with a given geometric
signature, we have the following result, which is a direct consequence of
\cite{book:lange22}*{Theorem~3.1.4}

\begin{proposition}
  \label{prop:RET}
  A Galois covering with geometric signature as in
  equation~\eqref{eq:geometric_signature} exists if and only if there is a
  \((2g+s)\)-tuple \((a_1,b_1,\ldots,a_g,b_g,c_1,\ldots,c_s)\) of elements of
  \(G\) such that the following conditions are satisfied:
  \begin{enumerate}
  \item \label{item:RET-gen} The elements
    \(a_1,b_1,\ldots,a_g,b_g,c_1,\ldots,c_s\) generate \(G\).
  \item \label{item:RET-conj} The group \(\gen{c_i}\) is conjugate to \(G_i\)
    for \(i=1,\ldots,s\).
  \item \label{item:RET-prod} We have
    \(\smallprod_{i=1}^g [a_i,b_i] \smallprod_{j=1}^s c_j = 1\), where
    \([a_i,b_i] = a_ib_ia_i^{-1} b_i^{-1}\).
  \end{enumerate}
\end{proposition}

\begin{definition}
  A tuple satisfying items~\ref{item:RET-gen} to~\ref{item:RET-prod} of
  Proposition~\ref{prop:RET} is called a \emph{generating vector of type}
  \((g; n_1,\ldots,n_s)\) of \(G\).
\end{definition}

\subsection{Group algebra components as Prym varieties}
\label{sec:prym-varieties}

For a compact Riemann surface \(X\), we denote its Jacobian variety by
\(\J(X)\); it is a principally polarized abelian variety. Given a (ramified)
covering map \(f\colon X \to Y\) between compact Riemann surfaces, the pullback
\(f^*\colon \J(Y) \to \J(X)\) is an \emph{isogeny} onto its image \(f^*\J(Y)\).

\begin{definition}
  The \emph{Prym variety} of \(f\) is the unique complementary abelian
  subvariety of \(f^*\J(Y)\) in \(\J(X)\) (with respect to its polarization),
  it is denoted by \(\pr(f)\).
\end{definition}

Now consider a Galois covering \(\pi_G\colon X\to Y\) as in the previous subsection
and consider two subgroups \(H\) and \(N\) of \(G\) such that \(H\subset N\). The
following result gives a decomposition of \(\J(X/H)\) and \(\pr(\pi^H_N)\) into a
product of group algebra components of \(\J(X)\) up to isogeny, we call it the
\emph{group algebra decomposition} of \(\J(X/H)\) and \(\pr(\pi^H_N)\),
respectively.

\begin{proposition}[\cite{book:lange22}*{Corollary~3.5.8~and~Corollary~3.5.9}]
  \label{prop:prym-decomp}
  With the previous notation, we have:
  \begin{enumerate}
  \item Set \(u_i \coloneq \inpr{\rho_H}{\varsigma_i} / \s(\varsigma_i)\) for \(i=1,\ldots,r\), then
    \[ \J(X/H) \sim \J(X/G) \times B_2^{u_2}\times \cdots \times B_r^{u_r}. \]
  \item \label{item:pr-decomp} Set
    \(t_i \coloneq \paren[\big]{\inpr{\rho_H}{\varsigma_i} - \inpr{\rho_N}{\varsigma_i}} / \s(\varsigma_i)\) for
    \(i = 2, \ldots, r\), then
    \begin{equation*}
      \pr(\pi^H_N) \sim B_2^{t_2} \times \cdots \times B_r^{t_r}.
    \end{equation*}
  \end{enumerate}
\end{proposition}

Since \(B_1 \sim \J(X/G)\) , we are concerned into describing the group algebra
components \(B_i\) for \(i=2,\ldots,r\) as Prym varieties. The following corollary
gives a characterization of components that are isogenous to Prym varieties of
intermediate coverings (see \cite{book:lange22}*{Corollary~3.5.10}).

\begin{corollary}
  \label{cor:rep_isolation}
  If \(\rho_H \cong \rho_N \oplus W_j\), then \(\pr(\pi^H_N) \sim B_i\).
\end{corollary}

For some values of \(i\) we may have \(\dim B_i = 0\), and hence
\(B_i \sim \pr(\pi_N^H)\) for any intermediate covering with trivial Prym variety
(in particular for \(H=N\)). In contrast, item~\ref{item:hom} of
Proposition~\ref{prop:isotypical} implies that there is no equivariant isogeny
between non-trivial components \(B_i\) and \(B_j\) if \(i \neq j\). Thereby, if in
the group algebra decomposition of \(\pr(\pi^H_N)\) there are at least two
different non-trivial group algebra components \(B_i\) and \(B_j\) with
\(t_i,t_j \geq 1\) or one non-trivial \(B_i\) with \(t_i\geq 2\), then
\(\pr(\pi^H_N)\) is not isogenous to any \(B_i\). This shows the importance of
knowing the dimension of each \(B_i\); it depends only on the geometric
signature of \(\pi_G\), in equation~\eqref{eq:geometric_signature}, and may be
computed as follows.

\begin{proposition}[\cite{book:lange22}*{Corollary~3.5.17}]
  \label{prop:dim-B_i}
  For \(i=2,\ldots,r\), the dimension of the group algebra component \(B_i\) is
  \[
    \dim B_i = \extdeg{L_i}{\QQ}\paren[\Big]{\deg(\varsigma_i) (g-1) + \frac{1}{2}
      \sum_{j=1}^s \paren[\big]{\deg(\varsigma_i) - \inpr{\rho_{G_j}}{\varsigma_i}}},
  \]
  where \(L_i\) denotes the field of definition of \(\varsigma_i\).
\end{proposition}

\begin{definition}
  \label{def:affordable}
  If each group algebra component of the Jacobian of a Galois covering is
  isogenous to the Prym variety of an intermediate covering, we say that its
  group algebra decomposition is \emph{affordable by Prym varieties}.
\end{definition}

\section{One dimensional affine group over a finite field}
\label{sec:affine-group}

For the rest of this article, set \(G \coloneq \GA{q}\), the group of affine
transformations of the Galois field \(\GF{q}\) of characteristic \(p\) with
\(q = p^n\). Each \(g\in G\) is a map of the form \(g(x) = ax+b\) with
\(b\in \GF{q}\) and \(a\in\GF{q}^\times\). Let \(\alpha \in \GF{q}^\times\) be a primitive
\((q-1)\)-root of unity, then \(\GF{q}^\times = \gen{\alpha}\) as a multiplicative group
(see \cite{book:lidl97}*{section~2.5}). Let \(\tau_b,\lambda_k \in G\) be the elements
defined by \(\tau_b(x) = x + b\) and \(\lambda_k(x) = \alpha^kx\) for
\(b\in \GF{q}\) and \(k=1,\ldots,q-1\); in particular, set
\(\tau \coloneq \tau_1\) and \(\lambda \coloneq \lambda_1\). Also set
\(\tra \coloneq \gen{\tau_b : b\in \GF{q}} \cong \GF{q}\) and
\(\lin \coloneq \gen{\lambda} \cong \GF{q}^\times\), the subgroups of translations and linear maps of
\(G\), respectively. Note that
\(\lambda_k\tau_b\lambda_k^{-1} = \tau_{\lambda_k(b)}\) and \(G = \tra \lin\), hence
\(\tra\) is a normal subgroup of \(G\) and \(G = \tra \rtimes \lin\).

\begin{remark}
  We will write \(\tau_b\lambda_k\) for an arbitrary element of \(G\). Note that
  \(\tau_0 = \lambda_{q-1}\), and it is the identity of \(G\), we will denote it by
  \(\id\).
\end{remark}

\subsection{Complex irreducible representations}
\label{sec:irrC}

Since \(\tra\) is abelian and \(G = \tra \rtimes \lin\), we will use the Wigner--Mackey method
of \emph{little groups} (see \cite{book:serre77}*{section~8.2}) in order to
determine the complex irreducible representations of \(G\), which, according to
the following proposition, should be \(q\) up to isomorphism.

\begin{proposition}
  \label{prop:conj-class}
  The group \(G\) has \(q\) different conjugacy classes, namely \([\tau]\) and
  \([\lambda_k]\) for \(k=1,\ldots, q-1\).
\end{proposition}

\begin{proof}
  For \(\tau_b\lambda_k \in G\) we have
  \begin{align*}
    (\tau_b\lambda_k) \tau (\tau_b\lambda_k)^{-1} &= \tau_b (\lambda_k \tau \lambda_k^{-1}) \tau_{-b} =
                               \tau_{\alpha^k},\\
    \intertext{hence \(\tau_a \in [\tau]\) for all \(a\in\GF{q}^\times\). Also}
    (\tau_b\lambda_k) \lambda_j (\tau_b\lambda_k)^{-1} &= \tau_{b-\lambda_j(b)} \lambda_k = \tau_{(1 - \alpha^j)b} \lambda_j;
  \end{align*}
  so, if \(j \neq q-1\), for any \(c \in \GF{q}\) we can set
  \(b \coloneq c/(1-\alpha^j)\) and hence
  \((\tau_b\lambda_k) \lambda_j (\tau_b\lambda_k)^{-1} = \tau_c \lambda_j\). Therefore
  \([\lambda_j] = \{\tau_c \lambda_j : c \in \GF{q}\}\) for
  \(j = 1, \ldots, q-2\). Finally, for \(j = q-1\), the class
  \([\lambda_{q-1}]\) is the conjugacy class of the identity.
\end{proof}

Let \(\zeta_k \in \CC^\times\) denote the primitive \(k\)-root of unity
\(\mathrm{e}^{2\ppi\mathrm{i}/k}\) for each positive integer \(k\). Since
\(\tra\) is naturally isomorphic to the additive group \(\GF{q}\), according to
\cite{book:lidl97}*{Theorem~5.7}, the irreducible characters of \(\tra\) are
\begin{alignat}{2}
  \label{eq:add_char}
  \chi_b(\tau_c) &= \zeta_{p}^{\tr(bc)} &\quad & \text{for \(b\in \GF{q}\);} \\
  \intertext{similarly, since \(\lin\) is naturally isomorphic to
  \(\GF{q}^\times\), according to \cite{book:lidl97}*{Theorem~5.8}, the
  irreducible characters of \(\lin\) are 
  \label{eq:mult_char}}
  \psi_j(\lambda_k) &= \zeta_{q-1}^{jk} && \text{for \(j = 1, \ldots, q-1\).}
\end{alignat}

\begin{theorem}
  \label{thm:comp-rep-aff}
  There are \(q-1\) complex irreducible representations of \(G\) of degree
  \(1\), they are extensions of the \(\psi_j\) in equation~\eqref{eq:mult_char},
  denoted also by \(\psi_j\), given by
  \begin{align}
    \label{eq:psi_rep}
    \psi_j(\tau_b\lambda_k) &\coloneq \zeta_{q-1}^{jk} \quad \text{for \(j = 1, \ldots, q-1\),} \\
    \intertext{and one of degree \(q-1\) defined by}
    \label{eq:standard_rep}
    \st &\coloneq \Ind_{\tra}^G(\chi_1),
  \end{align}
  where \(\chi_1\) is defined by equation~\eqref{eq:add_char} and \(\tra\) is the
  subgroup of translations of \(G\).

  These are all the complex irreducible representations of \(G\) up to
  isomorphism, and all of them have Schur index equal to \(1\).
\end{theorem}

\begin{proof}
  We follow the procedure described in \cite{book:serre77}*{section~8.2}. The
  group \(\lin\) acts on \(\Hom(\tra, \CC^*)\) by
  \begin{equation}
    (\lambda_k\chi_b)(\tau_c) = \chi_b(\lambda_k\tau_c \lambda_k^{-1}) = \chi_b(\tau_{\alpha^{k}(c)}) = \chi_{\alpha^{k}b}(\tau_c)
    = \chi_{\lambda_{k}(b)}(\tau_c),
  \end{equation}
  that is \(\lambda_k\chi_b = \chi_{\lambda_{k}(b)}\). Since \(\lin\) acts transitively on
  \(\GF{q}^\times\) and \(\lambda_k(0) = 0\), there are two \(\lin\)-orbits in
  \(\Hom(\tra, \CC^\times)\), namely \(\{\chi_0\}\) (the orbit of the trivial
  representation of \(\tra\)) and \(\{\chi_b : b \in \GF{q}^\times \}\). Hence
  \(\{\chi_0, \chi_1\}\) is a full system of representatives of the
  \(\lin\)-orbits in \(\Hom(\tra, \CC^\times)\). Set
  \(H_b \coloneq \Stab_\lin (\chi_b)\) for \(b=0,1\), then \(H_0 = \lin\) and
  \(H_1 = \{\id\}\); set \(G_0 \coloneq \tra H_0 = G\) and
  \(G_1 \coloneq \tra H_1 = \tra\). Each map \(\chi_b\) extends to \(G_b\) by
  \(\chi_b(th) = \chi_b(t)\) for \(t \in \tra\), \(h \in H_b\) and
  \(b = 0, 1\). We study the two cases separately.

  For \(b=0\), the complex irreducible representations of \(H_0 = \lin\) are
  described by equation~\eqref{eq:mult_char} which lift to the representations
  described by equation~\eqref{eq:psi_rep} by the quotient map \(G \to \lin\),
  also \(\Ind_G^G(\chi_0 \otimes \psi_j) = \psi_j\). Since \(\deg(\psi_j) = 1\) we have
  \(\s(\psi_j) = 1\). Note that \(\psi_{q-1} = 1_G\).

  For \(b=1\), the only complex irreducible representation of \(H_1 = \{\id\}\)
  is the trivial representation \(1_{H_1}\), which lifts to the trivial
  representation \(1_\tra\) by the quotient \(\tra \to \{\id\}\). Thus we obtain
  the irreducible representation
  \(\st \coloneq \Ind_\tra^G(\chi_1 \otimes 1_\tra) = \Ind_\tra^G(\chi_1)\) of
  equation~\eqref{eq:standard_rep}.

  According to \cite{book:serre77}*{Proposition~25}, those are all the
  irreducible representations of \(G\).

  The character of \(\st\) may be computed from
  \(\chi_\st(\id) = \deg(\st) = q-1\) and orthogonality relations for the rest of
  the conjugacy classes of \(G\); it is given in
  Table~\ref{table:complex_char_table}. Note that
  \(\chi_\st(g) = \card{\Fix_g(\GF{q})} - 1\), where \(\Fix_g(\GF{q})\) denotes
  the subset of \(\GF{q}\) fixed by \(g\), and hence \(\st\) is equivalent to
  the restriction of the standard representation of the symmetric group over
  \(\GF{q}\); therefore, the representation \(\st\) is realizable over the
  integers and its Schur index is \(1\).
\end{proof}

\begin{corollary}[Of the proof]
  The character table of \(G\) is given by
  Table~\ref{table:complex_char_table}.
\end{corollary}

\begin{table}
  \caption{Character table of \(\GA{q}\)}
  \label{table:complex_char_table}
  \begin{tabular}{@{} l *{6}{c} @{}} \toprule
    & \(1\) & \(q-1\) & \(q\) & \(q\) & & \(q\)\\
    & \(\id\) & \(\tau\) & \(\lambda\) & \(\lambda_2\) & \(\cdots\) & \(\lambda_{q-2}\) \\ \midrule
    \(1_G\) & \(1\) & \(1\) & \(1\) & \(1\) & \(\cdots\) & \(1\) \\
    \(\psi_1\) & \(1\) & \(1\) & \(\zeta_{q-1}\)& \(\zeta_{q-1}^2\)& \(\cdots\) & \(\zeta_{q-1}^{q-2}\) \\[1ex]
    \(\psi_2\) & \(1\) & \(1\) & \(\zeta_{q-1}^2\)& \(\zeta_{q-1}^4\)& \(\cdots\) & \(\zeta_{q-1}^{2(q-2)}\) \\
    \;\(\vdots\) & \(\vdots\) & \(\vdots\) & \(\vdots\) & \(\vdots\) & \(\ddots\) & \(\vdots\) \\
    \(\psi_{q-2}\) & \(1\) & \(1\) & \(\zeta_{q-1}^{q-2}\) & \(\zeta_{q-1}^{q-3}\) & \(\cdots\)
                                        & \(\zeta_{q-1}\)\\[1ex]
    \(\st\) & \(q-1\) & \(-1\) & \(0\) & \(0\) & \(\cdots\) & \(0\) \\ \bottomrule
  \end{tabular}
\end{table}

\subsection{Rational irreducible representations}
\label{sec:irrQ}

Now we describe the rational irreducible representations of \(G\) based on the
results of sections~\ref{sec:rational-rep}~and~\ref{sec:irrC}, we keep the
notation of the section~\ref{sec:irrC}. Let \(\nd(n)\) denote the number of
positive divisors of an integer \(n\).

\begin{proposition}
  \label{prop:rat-conj}
  The group \(G\) has \(\nd(q-1) + 1\) rational conjugacy classes. They are
  \([\tau]_{\QQ}\) and \([\lambda_d]_{\QQ}\) for \(d \mid q-1\).
\end{proposition}
\begin{proof}
  According to Proposition~\ref{prop:conj-class}, the conjugacy classes of
  \(G\) are \([\tau]\) and \([\lambda_k]\) for \(k=1,\ldots,q-1\). Note that two different
  class representatives generate conjugate cyclic subgroups of \(G\) if and
  only if they are \([\lambda_k]\) and \([\lambda_j]\) with
  \(\gcd(k,q-1) = \gcd(j,q-1)\). Hence, there is one rational conjugacy class
  for each divisor of \(q-1\), and \([\tau]_{\QQ} = [\tau]\).
\end{proof}

\begin{theorem}
  \label{thm:rat-rep-aff}
  There are \(\nd(q-1) + 1\) rational irreducible representations of \(G\) up
  to isomorphism, \(\nd(q-1)\) of them are given by
  \[
    \xi_d \coloneq \bigoplus_{\mathclap{\substack{ 1\leq k \leq q-1\\ \gcd(k,q-1)=d}}} \psi_k \quad
    \text{for \(d \mid q-1\).}
  \]
  And the other one is \(\st\) as defined by equation~\eqref{eq:standard_rep}.
\end{theorem}
\begin{proof}
  According to equation~\eqref{eq:rat-decomp}, a rational irreducible
  representation \(\xi_d\) Galois associated to \(\psi_d\) for \(d\mid q-1\) is given
  by \(\xi_d = \bigoplus_{\gamma \in \Gal(K/\QQ)}\psi_d^\gamma\) (recall from
  Theorem~\ref{thm:comp-rep-aff} that the Schur index of each complex
  irreducible representation of \(G\) is \(1\)), where \(K\) is the character
  field of \(\psi_d\), that is
  \(K = \QQ(\zeta_{q-1}^d) =
  \QQ(\zeta_{(q-1)/d})\). Table~\ref{table:complex_char_table} yields
  \(\psi_d^\gamma(\lambda) = \gamma(\zeta_{(q-1)/d})\). Since \(K\) is a cyclotomic field, the Galois
  group \(\Gal(K/\QQ)\) acts transitively on the set of primitive
  \((q-1)/d\)-roots of unity; hence, for \(k=1, \ldots, q-1\), there is a
  \(\gamma\in\Gal(K/\QQ)\) such that
  \(\gamma(\zeta_{(q-1)/d}) = \zeta_{q-1}^{k}\) if and only if
  \(\gcd(k,q-1) = d\). This yields the first part of the theorem. That \(\st\)
  is realizable over \(\QQ\) was already proven.
\end{proof}

For convenience, let \(d_1, \ldots, d_r\) denote the positive divisors of
\(q-1\) in decreasing order. With this notation \(\xi_{d_1} = \xi_{q-1}\) is
equivalent to \(1_G\).

\subsection{Subgroups}
\label{sec:subgroups}

In order to decompose the Prym varieties of the intermediate coverings of a
Galois covering \(\pi_G\colon X \to Y\) through Proposition~\Ref{prop:prym-decomp},
we give a characterization of the subgroups of \(G\) up to conjugacy. Recall
that, for each divisor \(m\) of \(n\), the field \(\mathbb{F}_{q}\) is an
extension of \(\GF{p^m}\) of degree \(n/m\) (see
\cite{book:lidl97}*{Theorem~2.6}). In particular, since \(\tra \cong \GF{q}\), it
is a \((n/m)\)-dimensional vector space over any field \(\GF{p^m}\) with
\(m \mid n\), where scalar multiplication is given by
\[
  a\cdot \tau_b = \tau_{ab} \quad \text{for all \(a\in \GF{p^m}\) and \(\tau_b \in \tra\)}.
\]

\begin{remark}
  \label{rem:mult_is_conj}
  Scalar multiplication by \(\alpha^k\) is equivalent to conjugacy by \(\lambda_k\).
\end{remark}
For cleaner notation, let \(\lin_d\) denote the unique subgroup of \(\lin\) of
order \(d\), namely \(\gen{\lambda_{(q-1)/d}}\).

\begin{theorem}
  \label{thm:subgroups}
  Each subgroup \(H\) of \(G\) is conjugate to a subgroup
  \(V\rtimes \lin_d \subset G\) where \(d\mid q-1\) and \(V\) is a
  \(\GF{p^m}\)-linear subspace of \(\tra\) with \(m\) the least positive
  divisor of \(n\) such that \(d \mid p^m -1\).
\end{theorem}
\begin{proof}
  Let \(H\) be a subgroup of \(G\) and set
  \(e \coloneq \min\{k : \lambda_k \in gHg^{-1} \cap \lin, g\in G\}\). Then, up to conjugacy, we
  may assume that \(\lambda_e \in H\), hence \(\gen{\lambda_e} \subset H\) and
  \(e \mid k\) for all \(\tau_b\lambda_k \in H\) (otherwise, the minimality of
  \(e\) would be contradicted by the division algorithm). Now set
  \(V \coloneq H \cap \tra\), then \(\tau_b\lambda_k \in H\) if and only if
  \(\tau_b = (\tau_b\lambda_k)\lambda_e^{k/e} \in V\), hence
  \(H = V\rtimes \gen{\lambda_e}\). Set \(d \coloneq (q-1)/e\), then
  \(H = V\rtimes \lin_d\). Also, the subgroup \(V\) must be invariant under
  conjugation by \(\lambda_e\), but, according to Remark~\ref{rem:mult_is_conj}, this
  is the same as being closed under scalar multiplication by \(\alpha^e\). Since
  \(\alpha^e \in \GF{q}^\times\) is a primitive \(d\)-root of unity, we have that
  \([\GF{p}(\alpha^e) : \GF{p}] = m\) where \(m\) is the least positive integer such
  that \(d \mid p^m - 1\) (see \cite{book:lidl97}*{Theorem~2.47}), moreover
  \(m \mid n\) and \(\GF{p}(\alpha^e) = \GF{p^m}\). Therefore \(V\) is closed under
  multiplications by scalars in \(\GF{p^m}\); namely, it is a
  \(\GF{p^m}\)-linear subspace of \(\tra\).

  We must prove that a set \(V\lin_d\) as above is indeed a subgroup of \(G\),
  but this follows directly from the following equation:
  \[
    (\tau_b\lambda_e^{k})(\tau_c\lambda_e^{j}) = \tau_{b+\alpha^{ke}c}\lambda_e^{kj} \quad \text{for all
      \(\tau_b,\tau_c\in V\) and \(k,j\in\ZZ\)}. \qedhere
  \]
\end{proof}

It is possible for two subgroups of the form \(H = V\rtimes \lin_d\) and
\(N = W\rtimes \lin_d\) with \(V\neq W\) to be conjugates, a necessary condition is that
\(\dim_{\GF{p^m}}(V) = \dim_{\GF{p^m}}(W)\), but it is not a sufficient
condition in general (a standard counting argument shows that there are
non-conjugate linear subspaces of the same dimension \(k\) if
\(2\leq k \leq n/m - 2\)). Nevertheless, according to Remark~\ref{rem:Frobenius}, the
decomposition of the representations \(\rho_H\) and \(\rho_N\) only depends on the
cardinality of the \(G\)-conjugacy classes of \(H\) and \(N\), respectively;
therefore, Proposition~\ref{prop:conj-class} implies that
\(\rho_N \cong \rho_H\) if and only if
\(\dim_{\GF{p^m}}(V) = \dim_{\GF{p^m}}(W)\). This motivates the following
definition.

\begin{definition}
  \label{def:type}
  The \emph{type} of a subgroup \(H\) of \(G\) conjugate to \(V\rtimes \lin_d\) with
  \(\dim_{\GF{p}}(V) = k\) is \((d,k)\).
\end{definition}

Note that for a fixed divisor \(d \mid q-1\), there are subgroups of \(G\) of type
\((d,k)\) if and only if \(k\) is a multiple of \(m\), as stated in
Theorem~\ref{thm:subgroups}. This assumption will be implicit each time we
discuss about a subgroup of a given type; we could have defined \(k\) as
\(\dim_{\GF{p^m}}(V)\) instead of \(\dim_{\GF{p}}(V)\) and then we would have
\(k\in\{1,\ldots,n/m\}\), but setting \(k \coloneq \dim_{\GF{p}}(V)\) keeps a cleaner and
more intuitive notation.

\begin{corollary}[of Theorem~\ref{thm:subgroups}]
  \label{cor:subgroups}
  \hfill
  \begin{enumerate}
  \item \label{item:order} The type of a subgroup \(H\) of \(G\) is determined
    by its order.
  \item \label{item:induced_rep} For two subgroups \(H\) and \(N\) of the same
    type we have \(\rho_H \cong \rho_N\).
  \item \label{item:subgroups} A subgroup of type \((d,k)\) has a subgroup of
    type \((e,j)\) if and only if \(e \mid d\) and \(j\leq k\).
  \end{enumerate}
\end{corollary}
\begin{proof}
  A subgroup \(H\) of \(G\) is conjugate to a subgroup \(V\rtimes \lin_d\). If
  \(k = \dim_{\GF{p}}(V)\), then \(\card{H} = p^k d\) with \(\gcd(p,d) =
  1\). Hence item~\ref{item:order} follows from the fundamental theorem of
  arithmetic.

  Item~\ref{item:induced_rep} was already discussed above.

  For item~\ref{item:subgroups}, necessity of the condition follows from
  item~\ref{item:order} and Lagrange's theorem. We now prove sufficiency of the
  condition. Let \(H\) be a subgroup of \(G\) of type \((d,k)\), then it is
  conjugate to \(V\rtimes \lin_d\) with \(V\) a \(\GF{p^m}\)-vector space. Since
  \(e \mid d\), we have \(\alpha^{(q-1)/e} \in \GF{p^m}\) and hence
  \(\GF{p^{\hat m}} = \GF{p}(\alpha^{(q-1)/e})\), where \(\hat m\) is the least
  positive divisor of \(n\) such that \(e \mid p^{\hat m} - 1\), is a subfield of
  \(\GF{p^m}\). Therefore \(V\) is also a \(\GF{p^{\hat m}}\)-vector space. Let
  \(W\) be a \(j\)-dimensional \(\GF{p^{\hat m}}\)-linear subspace of \(V\),
  then \(W\rtimes \lin_e\) is a subgroup of \(H\) of type \((e,j)\).
\end{proof}

\begin{example}
  \label{ex:subgroups}
  We give the lattice of subgroups of \(G\) up to conjugacy in the particular
  case where \(q = 9\). The divisors of \(3^2-1 = 8\) are \(1\), \(2\), \(4\)
  and \(8\), but \(1\) and \(2\) are also divisors of \(3^1 - 1 = 2\), hence
  the possible types of subgroups of \(G\) are \((1,0)\), \((1,1)\), \((1,2)\),
  \((2,0)\), \((2,1)\), \((2,2)\), \((4,0)\), \((4,2)\), \((8,0)\) and
  \((8,2)\). The lattice is the following:
  
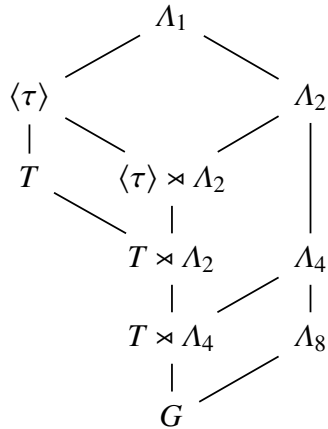
\begin{equation}
    \begin{tikzcd}[arrows=-, sep=small]
      & \lin_1 \ar[dl] \ar[dr] & \\
      \gen{\tau} \ar[d] \ar[dr] && \lin_2 \ar[dl] \ar[dd] \\
      \tra \ar[dr] & \gen{\tau} \rtimes \lin_2 \ar[d] & \\
      & \tra \rtimes \lin_2 \ar[d] & \lin_4 \ar[dl] \ar[d] \\
      & \tra \rtimes \lin_4 \ar[d] & \lin_8 \ar[dl] \\
      & G &
    \end{tikzcd}
  \end{equation}
\end{example}

\section{Galois covers with affine group action}
\label{sec:galois-covers}

Let \(\pi_G \colon X \to Y\) be a Galois covering with \(G\) and
\((g; G_1, \ldots, G_s)\) be its geometric signature. Each \(G_i\) must be a cyclic
subgroups of \(G\); hence, applying Definition~\ref{def:type} to
Proposition~\ref{prop:rat-conj}, they must be of type \((1, 1)\) or
\((d,0)\). Also, by item~\ref{item:order} of Corollary~\ref{cor:subgroups}, the
class of each \(G_i\) depends only on its order, thereby (in our particular
case) the geometric signature depends only on the signature
\begin{equation}
  \label{eq:signature}
  \paren[\big]{g; \underbrace{p,\ldots,p}_{\text{\(a\) times}},
    \underbrace{\vphantom{p}d_1,\ldots,d_1}_{\text{\(b_1\) times}}, \ldots,
    \underbrace{\vphantom{p}d_{r-1},\ldots,d_{r-1}}_{\text{\(b_{r-1}\) times}}},
\end{equation}
with \(a + \smallsum_{i=1}^{r-1} b_i = s\). There is no \(d_r\) in the
signature since \(d_r = 1\). In this section, we prove the main result of this
article; namely, that for \emph{most} values of \(q\) the group algebra
decomposition of \(\J(X)\) is not affordable by Prym varieties (see
Definition~\ref{def:affordable}). We will give a precise meaning of what
\emph{most} means in the last sentence. In this manner we obtain an infinite
family of Galois coverings with group algebra decomposition not affordable by
Prym varieties. Also, for the cases where the decomposition is affordable, we
compute it.

According to Theorem~\ref{thm:rat-rep-aff}, the group \(G\) has \(r\) rational
irreducible representations \(\xi_d\), for \(d \mid q-1\), and one rational
irreducible representation \(\st\), which is also absolutely irreducible and of
degree \(q-1\). Also, each \(\xi_d\) is Galois associated to the complex
irreducible representation \(\psi_d\) of degree \(1\).  Thereby, according to
Proposition~\ref{prop:ga-ab}, the group algebra decomposition of \(\J(X)\) is
of the form
\begin{equation}
  \label{eq:ga-aff}
  \J(X) \sim \paren[\Big]{\prod_{j=1}^r B_j} \times B_{r + 1}^{q-1},
\end{equation}
The group algebra component \(B_1\) corresponds to \(\xi_{d_1}\), the trivial
representation.

For the group algebra decomposition of \(\J(X)\) to be affordable by Prym
varieties, it is necessary for each non-trivial \(B_i\) in
equation~\eqref{eq:ga-aff} to be isogenous to the Prym variety of an
intermediate covering of \(\pi_G\). According to
Corollary~\ref{cor:rep_isolation}, for an intermediate covering \(\pi_N^H\), we
have \(B_i \sim \pr(\pi_N^H)\) if \(\rho_N = \xi_{d_i} \oplus \rho_H\) for
\(i=2,\ldots,r\) or \(\rho_N = \st \oplus \rho_H\) for \(i = r+1\). Moreover, according to
Corollary~\ref{cor:subgroups}, the representations \(\rho_H\) and \(\rho_N\) depend
only on the type of \(H\) and \(N\), respectively.

\begin{lemma}
  \label{lem:rho_H}
  For a subgroup \(H\) of type \((d,k)\) of \(G\), we have
  \[
    \rho_H \cong \paren[\Big]{\bigoplus_{d \mid \delta} \xi_\delta} \oplus \frac{p^{n-k} - 1}{d} \st.
  \]
\end{lemma}

\begin{proof}
  According to Theorem~\ref{thm:subgroups}, the subgroup \(H\) is conjugate to
  a subgroup \(V \rtimes \lin_d\) with \(\dim_{\GF{p}}(V) = k\), hence it has
  \(1\) element in the conjugacy class \([\id]\), \(p^k-1\) elements in
  \([\tau]\) and \(p^k\) in each \([\lambda_{e}]\) for which
  \(e \mid (q-1)/d\). By Remark~\ref{rem:Frobenius} and
  Table~\ref{table:complex_char_table} we have
  \begin{align*}
    \inpr{\rho_H}{\st} &= \frac{1}{p^kd}\paren[\big]{q-1 - (p^k-1)} = \frac{p^{n-k}
                    - 1}{d}\\
    \intertext{and, setting \(\hat d \coloneq \gcd(d,\delta)\),}
    \inpr{\rho_H}{\psi_\delta} &= \frac{1}{p^k d} \paren[\Big]{p^k \sum_{i=1}^d
                      \zeta_{q-1}^{\frac{q-1}{d} \delta i}} = \frac{1}{d} \sum_{i=1}^d
                      \paren[\big]{\zeta_d^{\hat d}}^{(\delta/\hat{d})i} \\
                  &= \frac{\hat d}{d} \sum_{i=1}^{d/\hat{d}} \zeta_{d/\hat{d}}^i =
                    \begin{cases*}
                      1 & if \(d = \hat{d}\), \\
                      0 & if \(d \neq \hat{d}\). \qedhere
                    \end{cases*}
  \end{align*}
\end{proof}

Lemma~\ref{lem:rho_H} and Corollary~\ref{prop:prym-decomp} directly imply the
following result.

\begin{corollary}
  \label{cor:Prym-decom}
  For two subgroups \(H\) and \(N\) of \(G\) with \(H\subset N\) of types
  \((e,j)\) and \((d,k)\), respectively, we have
  \[
    \pr(\pi^H_N) \sim \paren[\Big]{\prod_{\crampedsubstack{
          e \mid d_i \\
          d \nmid d_i}} B_i} \times B_{r+1}^s \quad \text{with
      \(s = \frac{p^{n-j} -1}{e} - \frac{p^{n-k}-1}{d}\).}
  \]
\end{corollary}

\begin{example}
  We present the lattice of example~\ref{ex:subgroups}, in which \(q=9\), with
  the decomposition of \(\rho_H\) instead of each subgroup \(H\) and
  \(\rho_H - \rho_N\) in the edges corresponding to a Prym variety isogenous to a
  single group algebra component:
  
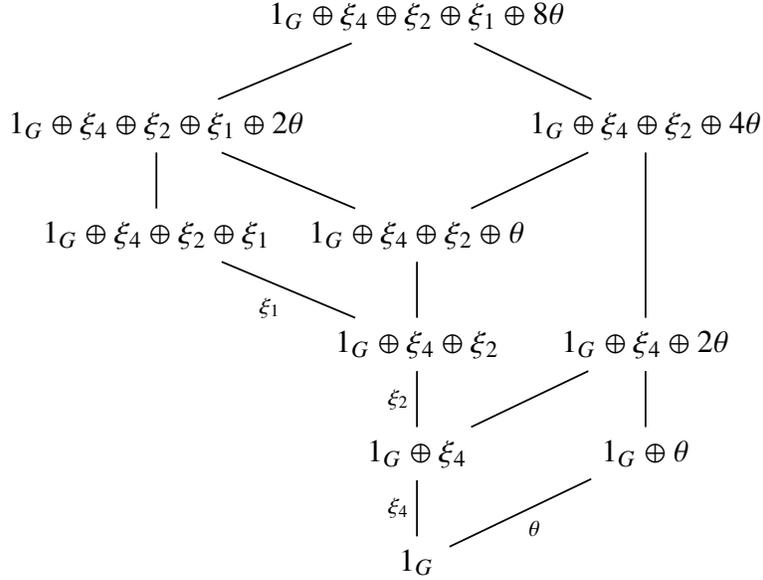
\begin{equation}
    \begin{tikzcd}[arrows=-, column sep=-2em]
      & 1_G \oplus \xi_4 \oplus \xi_2 \oplus \xi_1 \oplus 8\st \ar[dl]\ar[dr] & \\
      1_G \oplus \xi_4 \oplus \xi_2 \oplus \xi_1 \oplus 2\st \ar[d]\ar[dr] && 1_G \oplus \xi_4 \oplus \xi_2 \oplus 4\st
      \ar[dl]\ar[dd]\\
      1_G \oplus \xi_4 \oplus \xi_2 \oplus \xi_1 \ar[dr,"\xi_1"'] &  1_G \oplus \xi_4 \oplus \xi_2 \oplus \st \ar[d] & \\
      &  1_G \oplus \xi_4 \oplus \xi_2 \ar[d,"\xi_2"'] & 1_G \oplus \xi_4 \oplus 2\st \ar[dl] \ar[d] \\
      & 1_G \oplus \xi_4 \ar[d,"\xi_4"'] & 1_G \oplus \st \ar[dl,"\st"] \\
      & 1_G &
    \end{tikzcd}
  \end{equation}
\end{example}

Before proving our main result we must compute the dimension of the group
algebra components. Naturally \(\dim B_1 = g\), for the rest of the components
we have the following lemma (recall the notation of
equation~\eqref{eq:signature}).

\begin{lemma}
  \label{lem:dim-B_i}
  For \(i=2,\ldots,r\), the dimension of \(B_i\) is given by
  \begin{align*}
    \dim B_i &= \varphi\paren[\Big]{\frac{q-1}{d_i}} \paren[\Big]{g-1 + \frac{1}{2}\sum_{d_j \nmid d_i}
               b_j}, \\ 
    \intertext{where \(\varphi\) denotes the \emph{totient} function, and}
    \dim B_{r+1} &= (g-1)(q-1) + \frac{1}{2}\paren[\bigg]{ap^{n-1}(p-1) + (q-1)
                   \sum_{j=1}^{r-1} b_j\paren[\Big]{ 1 - \frac{1}{d_j}}}.
  \end{align*}
\end{lemma}

\begin{proof}
  Follows directly from Proposition~\ref{prop:dim-B_i} and
  Lemma~\ref{lem:rho_H}.
\end{proof}

Now we can prove our main result.
\begin{theorem}
  \label{thm:main}
  The group algebra decomposition of \(\J(X)\) is affordable by Prym varieties
  if and only if at least one of the following three conditions is met:
  \begin{enumerate}
  \item\label{item:q_con} The integer \(q\) is equal to \(2\) or \(9\), is a
    Fermat prime or \(q-1\) is a Mersenne prime.
  \item\label{item:sign_con} The signature of \(\pi_G\) is of the form
    \begin{equation}
      \label{eq:sign_con}
      (1;p,\ldots,p)\ \text{or}\ (0;p,\ldots,p,q-1,q-1).
    \end{equation}
  \item\label{item:mix_con} The integer \(q-1\) is equal to \(d^\mu e^\nu\), where
    \(d\) and \(e\) are different prime numbers and \(\mu\) and \(\nu\) are
    positive integers, and the signature of \(\pi_G\) is of the form
    \begin{equation}
      \label{eq:mix_con}
      \begin{gathered}
        (1;p,\ldots,p,d,\ldots,d), (0;p,\ldots,p,q-1,q-1,d,\ldots,d), \\
        (0;p,\ldots,p,q-1, e^\nu, d, \ldots, d)\ \text{or}\ (0;p,\ldots,p,e^\nu, e^\nu, d, \ldots, d)
      \end{gathered}
    \end{equation}
    (the last two signatures are only possible
    if \(\mu = 1\)).
  \end{enumerate}
\end{theorem}

\begin{proof}
  We first prove necessity of the conditions in items~\ref{item:q_con}
  to~\ref{item:mix_con}. Hence, assume that the group algebra decomposition of
  \(\J(X)\) is affordable through Prym varieties.

  Suppose that item~\ref{item:q_con} is not met, thus \(q - 1\) is neither
  \(1\) nor a prime number nor a power of a prime number. Indeed:
  \begin{itemize}
  \item If \(q-1 = 1\), then \(q = 2\).
  \item If \(q-1\) is prime, then either \(q\) is odd and then equal to \(3\),
    a Fermat prime, or \(q\) is even and then \(q = 2^n\) and \(q-1\) is a
    Mersenne prime.
  \item If \(p^n-1 = d^\nu\) for \(d\) prime and \(\nu \geq 2\), then we have two
    cases: if \(n=1\), then \(p\) is odd and \(p = 2^\nu+1\), a Fermat prime; if
    \(n > 1\), according to Mihăilescu's theorem (Catalan's conjecture, see
    \cite{art:mihailescu2004}*{Theorem~5}), we have \(q = 9\).
  \end{itemize}
  Therefore \(q-1\) is divisible by at least to different prime numbers. Since
  the group algebra decomposition of \(\J(X)\) is affordable by Prym varieties,
  either \(\dim B_r = 0\) or there are subgroups \(H\) and \(N\) of \(G\) such
  that \(\pr(\pi_N^H) \sim B_r\).

  Suppose that \(\dim B_r = 0\). Lemma~\ref{lem:dim-B_i} implies that
  \(\dim B_r = 0\) if and only if
  \(g-1 + \frac{1}{2}\smallsum_{j=1}^{r-1} b_j = 0\). The latter implies that
  either \(g = 1\) and the signature of \(\pi_G\) is as in
  item~\ref{item:sign_con} or \(g = 0\) and
  \(\frac{1}{2}\smallsum_{j=1}^{r-1} b_j = 2\); namely, there are just two
  integers different of \(p\) in the signature of \(\pi_G\) (see
  equation~\eqref{eq:signature}). Consider a generating vector
  \((c_1,\ldots,c_s)\) of \(G\) with \(g=0\) and just two elements out of
  \(\tra\), say \(c_1\) and \(c_2\). Item~\ref{item:RET-prod} of
  Proposition~\ref{prop:RET} implies that
  \(c_1^{-1} = c_2\smallprod_{i=3}^s c_i\), hence \(c_1\) and \(c_2\) generate
  conjugates of the same subgroup \(\lin_d\) of \(G\). Item~\ref{item:RET-gen}
  of Proposition~\ref{prop:RET} states that the elements \(c_1,\ldots, c_r\) must
  generate \(G\); therefore \(c_1\) and \(c_2\) must generate conjugates of
  \(\lin\) and hence their order is \(q-1\). This implies
  item~\ref{item:sign_con}.

  Now suppose that \(\dim B_r \geq 1\). In the notation of
  Corollary~\ref{cor:Prym-decom}, for a Prym \(\pr(\pi_N^H)\) to contain
  \(B_r\) it is necessary for \(H\) to be of type \((1,j)\) and for \(N\) to be
  of type \((e,k)\) with \(e > 1\). Since we are not interested in \(B_{r+1}\),
  set \(k \coloneq n\) and \(j \coloneq n\), so Corollary~\ref{cor:Prym-decom} yields
  \begin{equation}
    \label{eq:B_i-as-pr}
    \pr(\pi_N^H) \sim \prod_{e \nmid d_i} B_i.
  \end{equation}
  Consider a multiple \(\delta\) of \(e\), if \(e \nmid d_i\), then
  \(\delta \nmid d_i\); hence each \(B_i\) contained in \(\pr(\pi_N^H)\) would also be
  contained if \(N\) were of type \((\delta,n)\) instead of \((e,n)\). Thereby, for
  the left hand side of equation~\eqref{eq:B_i-as-pr} to have the least number
  of components possible, we assume that \(e\) is a prime number. Since \(q-1\)
  is not a power of a prime, it is not possible that all divisors \(d_i\) of
  \(q-1\) but \(1\) are divisible by \(e\); that is, there is at least one
  divisor \(d_i\) of \(q-1\) such that \(e \nmid
  d_i\). Equation~\eqref{eq:B_i-as-pr} implies that \(\dim B_i = 0\) for all
  \(i=2,\ldots, r-1\) such that \(e \nmid d_i\). We separate three cases depending on
  the genus of \(Y\):
  \begin{itemize}
  \item According to Lemma~\ref{lem:dim-B_i}, if \(g \geq 2\), then \(B_i\) is
    non-trivial for all \(i=2,\ldots,r\). A contradiction.
  \item If \(g=1\), then \(\smallsum_{j=1}^{r-1} b_j \geq 1\); hence there is at
    least one \(b_j \geq 1\), say \(b_k\). Set \(d \coloneq d_k\),
    Lemma~\ref{lem:dim-B_i} implies that \(\dim B_i> 0\) for all \(i\) such
    that \(d \nmid d_i\); hence \(d \mid d_i\) for all \(i=2,\ldots,r\) such that
    \(e \nmid d_i\). Thus \(q-1 = d^\mu e^\nu\) for positive integers \(\mu\) and
    \(\nu\), and \(d\) must be a prime different from \(e\). We now prove that
    \(b_k\) is the only positive \(b_j\). Suppose that \(b_{k'} \geq 1\). Set
    \(d' \coloneq d_{k'}\), then \(d' \mid d_i\) for all \(i=2,\ldots,r\) such that
    \(e \nmid d_i\); in particular \(d' \mid d\), hence \(d' = 1\) or
    \(d' = d\). But \(d'\) is the order of a ramification point, so \(d' = d\)
    and \(k' = k\). This implies item~\ref{item:mix_con}.
  \item If \(g=0\), then \(\smallsum_{j=1}^{r-1} b_j \geq 3\). Set \(\delta\) such that
    \(q-1 = \delta e^\nu\) with \(e \nmid \delta\). For each prime divisor
    \(d_i\) of \(\delta\) we have \(\dim B_i = 0\), hence
    \(\smallsum_{d_j \nmid d_i} b_j = 2\), set \(d\coloneq d_i\). This implies that
    \(\smallsum_{e^\nu \mid d_j} b_j \leq 2\), since no multiple of \(e^\nu\) divides
    \(d\). We now prove that the sum \(\smallsum_{e^\nu \mid d_j} b_j\) is exactly
    \(2\). Consider a generating vector \((c_1,\ldots,c_s)\) of \(G\) with
    \(g=0\). Item~\ref{item:RET-gen} of Proposition~\ref{prop:RET} implies that
    \(\gen{c_1,\ldots,c_n} = G\), hence
    \(\gen{c_1,\ldots,c_n} \tra / \tra \cong \lin\). Besides, we have
    \(\gen{c_1,\ldots,c_n} \tra / \tra \cong \lin_{\lcm_{1\leq i\leq s}\{\ord{c_i}\}}\); thus
    \(\lcm_{1\leq i \leq s}\{\ord{c_i}\} = q-1\) and at least one \(c_i\) in
    \(\{c_1,\ldots,c_n\}\) has order a multiple of \(e^\nu\), say
    \(c_1\). Item~\ref{item:RET-prod} of Proposition~\ref{prop:RET} implies
    that \(c_1^{-1} = c_2\smallprod_{i=3}^s c_i\), so \(\{c_2,\ldots,c_n\}\)
    generates \(G\); so, by a similar argument, we can assume that \(c_2\) has
    order a multiple of \(e^\nu\). Therefore
    \(\smallsum_{e^\nu \mid d_j} b_j = 2\). Since
    \(\smallsum_{j=1}^{r-1} b_j \geq 3\), there must be another \(b_j >
    1\). Since\(\smallsum_{e^\nu \mid d_j} b_j = 2\), it is necessary that
    \(b_j = 0\) for all \(j\) but the one or two already described, for which
    \(d_j\) is a multiple of \(e^\nu\) (\(c_1\) and \(c_2\) may have the same
    order), and the one for which \(d_j = d\). Since \(d\) was chosen
    arbitrarily, the same is true for every other prime divisor of \(\delta\),
    therefore \(d\) must be the only one, namely \(\delta = d^\mu\). We have
    \(\gen{c_1,c_3,\ldots,c_s}=\gen{c_2,c_3,\ldots,c_s} = G\) and
    \(\lcm_{3\leq i \leq s}\{\ord{c_i}\} \mid dq\), hence, if \(\mu > 1\), the elements
    \(c_1\) and \(c_2\) must have order \(d^\mu e^\nu = q-1\). If \(\mu = 1\), then
    \(c_1\) and \(c_2\) may have order \(e^\nu\) or \(de^\nu = q-1\). This implies
    item~\ref{item:mix_con}.
  \end{itemize}

  Now we prove the sufficiency of the conditions in
  items~\ref{item:q_con}~to~\ref{item:mix_con}. Theorem~\ref{cor:Prym-decom}
  implies that \(B_{r+1} \sim \pr(\pi_G^{\lin})\). Also, we have
  \(B_1 \sim \J(Y)\). This two isogenies yield the group algebra decomposition by
  Prym varieties for \(q=2\) and for the cases in item~\ref{item:sign_con}, for
  which \(B_i\) is trivial for all \(i=2,\ldots,r\); namely
  \begin{equation}
    \label{eq:simple-decomp}
    \J(X) \sim \J(Y) \times \pr(\pi_G^{\lin}).
  \end{equation}
  For all other cases in items~\ref{item:q_con} and~\ref{item:mix_con}, there
  are non-trivial components \(B_i\) for \(i\in\{2,\ldots,r\}\). In
  Table~\ref{tab:decomposition}, we give subgroups \(H_i\) and \(N_i\) such
  that \(H_i\subset N_i\) and \(B_i \sim \pr\paren[\big]{\pi_{N_i}^{H_i}}\) for each
  \(i\in\{2,\ldots,r\}\) with non-trivial \(B_i\); the computations follow directly
  from Theorem~\ref{cor:Prym-decom}. For simplicity, the trivial components are
  omitted. \qedhere

  \begin{table}
    \caption{Group algebra decomposition of \(\J(X)\) by Prym varieties.}
    \label{tab:decomposition}
    \begin{tabular}{@{} *6l  @{}}
      \toprule
      \(q\) & Signature & Decomposition & \(d_i\) & \(H_i\) & \(N_i\)\\
      \midrule
      \(9\) & any & \(B_1 \times B_2 \times B_3 \times B_4 \times B_5^8\)
                                        & \(4\) & \(\tra \rtimes \lin_4\)
                                                            & \(G\)\\
            &&& \(2\) & \(\tra \rtimes \lin_2\) & \(\tra \rtimes \lin_4\)\\
            &&& \(1\) & \(\tra\) & \(\tra \rtimes \lin_2\)\\
      \addlinespace
      \(2^\nu+1\) & any & \(\paren{\smallprod_{i=1}^{\nu+1} B_i} \times B_{\nu+2}^{2^\nu}\)
                                        & \(2^k\) &
                                                    \(\tra\rtimes\lin_{2^k}\)
                                                            &
                                                              \(\tra\rtimes\lin_{2^{k+1}}\)\\
      \addlinespace
      \(2^n\) & any & \(B_1\times B_2 \times B_3^{2^n-1}\) & \(2^n-1\) & \(\tra\) &
                                                                          \(G\)\\
      \addlinespace
      \(d^\mu e^\nu+1\) & eq.~\eqref{eq:mix_con} & \(B_1 \times \paren{\smallprod_{d_i \mid
                                               e^\nu} B_i} \times
                                               B_{\mu\nu+\mu+\nu+2}^{d^\mu e^\nu}\)
                                        & \(e^k\) & \(\tra\rtimes\lin_{e^k}\) &
                                                                          \(\tra\rtimes\lin_{e^{k+1}}\)\\
      \bottomrule
    \end{tabular}
  \end{table}
\end{proof}

\begin{corollary}[Of the proof]
  \label{cor:MT}
  If the group algebra decomposition of \(\J(X)\) is affordable by Prym
  varieties, then its decomposition is given in Table~\ref{tab:decomposition}
  or by equation~\eqref{eq:simple-decomp}.
\end{corollary}

\bibliography{references}

\begin{bibdiv}
\begin{biblist}

\bib{art:carocca2006}{article}{
      author={Carocca, Angel},
      author={Rodr\'{\i}guez, Rub\'{\i}~E.},
       title={Jacobians with group actions and rational idempotents},
        date={2006},
        ISSN={0021-8693,1090-266X},
     journal={J. Algebra},
      volume={306},
      number={2},
       pages={322\ndash 343},
         url={https://doi.org/10.1016/j.jalgebra.2006.07.027},
      review={\MR{2271338}},
}

\bib{art:faddeev1976}{article}{
      author={Faddeev, D.~K.},
       title={Complex representations of the full affine group over a finite
  field},
        date={1976},
        ISSN={0002-3264},
     journal={Dokl. Akad. Nauk SSSR},
      volume={230},
      number={2},
       pages={295\ndash 297},
      review={\MR{417270}},
}

\bib{art:iranmanesh1997}{article}{
      author={Iranmanesh, A.},
       title={Characters of the affine group in low dimensions},
        date={1997},
        ISSN={0129-2021},
     journal={Southeast Asian Bull. Math.},
      volume={21},
      number={1},
       pages={27\ndash 62},
      review={\MR{1474929}},
}

\bib{art:lange2004}{article}{
      author={Lange, Herbert},
      author={Recillas, Sevin},
       title={Prym varieties of pairs of coverings},
        date={2004},
        ISSN={1615-715X,1615-7168},
     journal={Adv. Geom.},
      volume={4},
      number={3},
       pages={373\ndash 387},
         url={https://doi.org/10.1515/advg.2004.022},
      review={\MR{2071812}},
}

\bib{book:lange22}{book}{
      author={Lange, Herbert},
      author={Rodr\'{\i}guez, Rub\'{\i}~E.},
       title={Decomposition of {J}acobians by {P}rym varieties},
      series={Lecture Notes in Mathematics},
   publisher={Springer, Cham},
        date={[2022] \copyright 2022},
      volume={2310},
        ISBN={978-3-031-10144-1; 978-3-031-10145-8},
         url={https://doi.org/10.1007/978-3-031-10145-8},
      review={\MR{4559690}},
}

\bib{book:lidl97}{book}{
      author={Lidl, Rudolf},
      author={Niederreiter, Harald},
       title={Finite fields},
     edition={Second},
      series={Encyclopedia of Mathematics and its Applications},
   publisher={Cambridge University Press, Cambridge},
        date={1997},
      volume={20},
        ISBN={0-521-39231-4},
        note={With a foreword by P. M. Cohn},
      review={\MR{1429394}},
}

\bib{art:mihailescu2004}{article}{
      author={Mih\u{a}ilescu, Preda},
       title={Primary cyclotomic units and a proof of {C}atalan's conjecture},
        date={2004},
        ISSN={0075-4102,1435-5345},
     journal={J. Reine Angew. Math.},
      volume={572},
       pages={167\ndash 195},
         url={https://doi.org/10.1515/crll.2004.048},
      review={\MR{2076124}},
}

\bib{art:moraga23}{article}{
      author={Moraga, Benjam{\'i}n~M.},
       title={Galois closure of a fivefold covering and decomposition of its
  jacobian},
        date={2023},
        ISSN={1531-586X},
     journal={Transform. Groups},
         url={https://doi.org/10.1007/s00031-023-09827-y},
}

\bib{art:recillas1974}{article}{
      author={Recillas, Sevin},
       title={Jacobians of curves with {$g\sp{1}\sb{4}$}'s are the {P}rym's of
  trigonal curves},
        date={1974},
     journal={Bol. Soc. Mat. Mexicana (2)},
      volume={19},
      number={1},
       pages={9\ndash 13},
      review={\MR{480505}},
}

\bib{book:serre77}{book}{
      author={Serre, Jean-Pierre},
       title={Linear representations of finite groups},
     edition={French},
      series={Graduate Texts in Mathematics},
   publisher={Springer-Verlag, New York-Heidelberg},
        date={1977},
      volume={Vol. 42},
        ISBN={0-387-90190-6},
      review={\MR{450380}},
}

\bib{art:siegel1992}{article}{
      author={Siegel, Eli~A.},
       title={On the characters of the affine group over a finite field},
        date={1992},
        ISSN={0021-8693,1090-266X},
     journal={J. Algebra},
      volume={151},
      number={1},
       pages={86\ndash 104},
         url={https://doi.org/10.1016/0021-8693(92)90133-7},
      review={\MR{1182016}},
}

\end{biblist}
\end{bibdiv}
\end{document}